\documentclass[a4paper]{amsart}
\usepackage{amsmath,amssymb,amsthm}
\usepackage{amscd}
\usepackage{array,enumerate}
\newtheorem{Thm}{Theorem}[section]
\newtheorem{Prop}[Thm]{Proposition}
\newtheorem{Lem}[Thm]{Lemma}
\newtheorem{Cor}[Thm]{Corollary}
\newtheorem{Thmint}{Theorem}[section]
\newtheorem{Propint}[Thmint]{Proposition}

\theoremstyle{definition}
\newtheorem{Rem}[Thm]{Remark}

\newcommand{\Cs}{C$^\ast$}
\newcommand{\id}{\mbox{\rm id}}
\newcommand{\Homeo}{\mathop{{\rm Homeo}}}

\newcommand{\rc}{\mathop{\rtimes _{\mathrm r}}}
\newcommand{\rca}[1]{\mathop{\rtimes _{{\mathrm r}, #1}}}
\newcommand{\fc}[1]{\mathop{\rtimes _{#1}}}

\newcommand{\ad}{\mathop{\rm Ad}}
\title[\Cs-crossed product norms]{Simple equivariant \Cs-algebras whose full and reduced crossed products coincide}
\author{Yuhei Suzuki}
\address{Graduate school of mathematics, Nagoya University, Chikusaku, Nagoya, 464-8602, Japan}
\email{yuhei.suzuki@math.nagoya-u.ac.jp}
\begin{document}
\begin{abstract}
For any second countable locally compact group $G$, we construct a simple $G$-C$^\ast$-algebra
whose full and reduced crossed product norms coincide.
We then construct its $G$-equivariant representation on another simple $G$-\Cs-algebra without the coincidence condition.
This settles two problems posed by Anantharaman-Delaroche in 2002.
Some constructions involve the Baire category theorem.
\end{abstract}
\maketitle
\section{Introduction}\label{Sec:intro}
Let $G$ be a locally compact group acting on a \Cs-algebra $A$.
Then one can form two canonical \Cs-completions of the twisted convolution algebra $C_{\mathrm c}(G, A)$, the full crossed product $A \rtimes G$ and the reduced crossed product $A \rc G$, defined via the universal and the regular covariant representation respectively. We refer the reader to \cite{BO} and \cite{Ped} for
details and basic facts on the crossed product construction.
Since both completions have their own advantages, it is plausible to find a reasonable condition characterizing the coincidence of these two completions.
Note that for group algebras (i.e., the case $A=\mathbb{C}$),
Hulanicki \cite{Hul} has shown that amenability characterizes the coincidence of the two completions.
Motivated by Hulanicki's theorem, it is natural to attempt to
find an appropriate notion of amenability for group actions on \Cs-algebras which characterizes the coincidence of the two crossed products.

In the seminal paper \cite{Ana}, Anantharaman-Delaroche posed a few problems to pursue what is the genuine definition of amenability of group actions on \Cs-algebras.
In this paper, we negatively answer the following questions (c) and (d) stated in Section 9.2 of \cite{Ana}.
\begin{enumerate}[\upshape(a)]\setcounter{enumi}{2}
\item
{\it Is the property $A \rtimes G = A \rc G$ functorial in an appropriate sense?
The strongest reasonable functoriality property is the following one: given any
$G$-\Cs-algebra $B$ and any non-degenerate $G$-representation $\Phi\colon A\rightarrow M(B)$ satisfying $\Phi(Z(M(A)))\subset Z(M(B))$, does the equality 
$A \rtimes G= A \rc G$ imply that $B \rtimes G= B \rc G$?}
\item
{\it If $A$ is a simple $G$-\Cs-algebra, does the equality $A \rtimes G = A \rc G$ imply that $G$ is amenable?}
\end{enumerate}
Here and throughout the paper, the equation $A \rtimes G = A \rc G$ means
the coincidence of the two crossed product norms.
Otherwise we denote $A \rtimes G \neq A \rc G$.
For the question (c), we note that amenability of group actions on locally compact Hausdorff spaces (Definition 2.1 of \cite{Ana}) has such a functorial property; indeed amenability inherits to extensions.
(Here recall that for two actions $\alpha$, $\beta$ of $G$ on compact spaces $X$, $Y$ respectively, $\alpha$ is said to be an extension of $\beta$
if there is a $G$-equivariant quotient map
from $X$ onto $Y$.)

More precisely, we prove the following theorem.
Since simple \Cs-algebras have no central multipliers other than scalars,
this settles both questions (c) and (d).
\begin{Thmint}\label{Thmint:main}
Let $G$ be a second countable locally compact non-amenable group.
Then there are simple $G$-\Cs-algebras $A, B$ and
a non-degenerate $G$-embedding $\Phi\colon A\rightarrow M(B)$
satisfying
\[A \rtimes G = A \rc G,\]
\[B \rtimes G \neq B \rc G.\]
\end{Thmint}
We note that our proof of Theorem \ref{Thmint:main} involves the Baire category theorem.
We also remark that \Cs-algebras constructed in the proof of Theorem \ref{Thmint:main} are all non-unital.
We do not know if the analogous statements hold true in the category of unital simple $G$-\Cs-algebras.
Still, when $G$ is discrete and exact, one can settle it as follows.
\begin{Propint}\label{Propint:unitalsimple}
Let $G$ be a discrete non-amenable exact group.
Then there is an inclusion of unital simple $G$-\Cs-algebras $A \subset B$
with a $G$-conditional expectation
satisfying
\[A \rtimes G = A \rc G,\]
\[B \rtimes G \neq B \rc G.\]
\end{Propint}
As a byproduct of our constructions in Theorem \ref{Thmint:main} and Proposition \ref{Propint:unitalsimple},
one can build the Cuntz algebra $\mathcal{O}_2$ and its stabilization $\mathcal{O}_2 \otimes \mathbb{K}$ as the full crossed products of actions of all possible groups
on themselves, which we believe is of independent interest. See Remark \ref{Rem:Cuntz} for a detailed explanation.

We end this paper by observing the following partial result in the unital setting.
This in particular settles question (c) even in the unital case for locally compact exact groups.
\begin{Propint}\label{Propint:functorialu}
Let $G$ be a locally compact non-amenable exact group.
Then there is an inclusion $A \subset B$ of unital $G$-\Cs-algebras
satisfying the following conditions.
\begin{enumerate}[\upshape (1)]
\item $Z(A)=Z(B)=\mathbb{C}$.
\item The inclusion $A \subset B$ admits a $G$-conditional expectation.
\item $A \rtimes G= A \rc G$, while $B \rtimes G \neq B \rc G$.
\end{enumerate}
\end{Propint}
We note that, unlike Theorem \ref{Thmint:main}, the proofs of Propositions \ref{Propint:unitalsimple} and \ref{Propint:functorialu}
avoid the Baire category theorem.

These statements clarify that $G$-actions on both the algebra of central multipliers and the structure space of the multiplier algebras
in general fail to detect ``amenability'' of actions on \Cs-algebras.
We note that an approach to amenability of actions via $Z(A^{\ast \ast})$, the center
of the second dual, is also studied in \cite{Ana} for discrete groups.

We close the introduction by summarizing recent progress related to the subject of the paper.
Matsumura \cite{Mat} has shown that for discrete exact groups,
amenablility of actions on compact Hausdorff spaces are characterized by
the coincidence of the two crossed products.
In particular the functorial property asked in question (c) in Section 9.2 of \cite{Ana}
holds true for actions of these groups on unital commutative \Cs-algebras.
Willett \cite{Wil} has constructed the first examples of (neither minimal nor essentially principal) non-amenable {\'e}tale groupoids whose full and reduced groupoid \Cs-algebras coincide.
Examples of principal groupoids with these conditions have been constructed by
Alekseev--Finn-Sell \cite{AFS}.
Brodzki--Cave--Li \cite{BCL} have established conjectured characterizations of exactness for
locally compact groups (cf.~\cite{Ana}, \cite{GK}, \cite{Oza}).

\subsection*{Notation}
\begin{itemize}
\item
For a $G$-\Cs-algebra $A$,
denote its full and reduced crossed products by $A \rtimes G$ and $A\rc G$ respectively.
When we need to specify the underlying action $\alpha$,
we denote them by $A \fc{\alpha} G$ and $A \rca{\alpha}G$ respectively.
\item The symbol `$\otimes$' stands for the minimal tensor products of \Cs-algebras and of completely positive maps, the diagonal actions of group actions, and the (Hilbert) tensor products of Hilbert spaces.
For an object $X$ of one of these categories, for $n\in \mathbb{N}$, for a set $S$,
we denote $\bigotimes_{i=1}^n X$
and $\bigotimes_{s\in S} X$ (if it makes a sense) by $X^{\otimes n}$ and $X^{\otimes S}$ respectively.
\item For a locally compact group $G$, denote $L^2(G)$ the $L^2$-space on $G$
with respect to the left Haar measure.
\item We identify two crossed products when there is a canonical isomorphism.
For instance, for a $(G \times H)$-\Cs-algebra $A$,
we denote
\[(A \rtimes G) \rtimes H = (A \rtimes H) \rtimes G=A \rtimes (G \times H).\]
Here the $H$-action on $A \rtimes G$ and the $G$-action on $A \rtimes H$
are the actions induced from the original action on $A$.
\item For a Hilbert space $\mathfrak{H}$, denote $\mathbb{K}(\mathfrak{H})$ (resp.~ $\mathbb{B}(\mathfrak{H})$) 
the \Cs-algebras of all compact (resp.~ bounded) operators on $\mathfrak{H}$. 
When $\mathfrak{H}=\ell^2(\mathbb{N})$, we simply denote $\mathbb{K}(\mathfrak{H})$ by $\mathbb{K}$.
\item For a \Cs-algebra $A$, denote $M(A)$ the multiplier algebra of $A$.
\item For a \Cs-algebra $A$, denote $Z(A)$ the center of $A$.
\item For a unital \Cs-algebra $A$, denote $\mathcal{U}(A)$ the unitary group of $A$.
\end{itemize}
\section{Generic properties of Cantor systems of infinite rank free group}
In this section we recall and establish a few generic properties of Cantor systems of an infinite rank free group. They become the key of the proof of Theorem \ref{Thmint:main}.

Let $X$ denote the Cantor set $\{0, 1\}^\mathbb{N}$.
We fix a (complete) metric $d_X$ on $X$.
On the homeomorphism group $\Homeo(X)$ of $X$,
we define a metric $d$ by
\[d(\varphi, \psi):= \max_{x\in X}d_X(\varphi(x), \psi(x)) +\max_{x\in X}d_X(\varphi^{-1}(x), \psi^{-1}(x)).\]
Then it is not hard to check that the $d$ is complete.
The induced topology
is equal to the uniform convergence topology (in particular it does not depend on the choice of $d_X$). This topology makes $\Homeo(X)$ a Polish group
(though we do not use the separability).

Fix two disjoint countable infinite sets $S$ and $T$.
Let $\Gamma$ denote the free group over the disjoint union
$U:=S\sqcup T$.
Denote $\Gamma_S$, $\Gamma_T$ the (free) subgroup of $\Gamma$
generated by $S, T$ respectively.
We then consider the set \[\mathcal{S}=\mathcal{S}(\Gamma, X):= \mathop{{\rm Hom}}(\Gamma, \Homeo(X))\]
of actions of ${\Gamma}$ on $X$.
To apply the Baire category theorem to $\mathcal{S}$,
we topologize $\mathcal{S}$ as follows.
(This is a standard approach to construct dynamical systems of new features;
we refer the reader to the introduction of \cite{Hoc} for information on this topic.)
The set $\mathcal{S}$ is a closed subset of the product space
$\Homeo(X)^{\Gamma}$, where the latter space is equipped with the product topology.
As $\Gamma$ is countable, the space $\Homeo(X)^\Gamma$
is Polish. The relative topology then defines a Polish space structure on $\mathcal{S}$.
Throughout the paper, we equip $\mathcal{S}$ with this topology.

For a property $\mathcal{P}$ of dynamical systems,
we say that $\mathcal{P}$ is open (resp.~G$_\delta$, dense, generic)
in $\mathcal{S}$
if the subset of $\mathcal{S}$ consisting of all elements with $\mathcal{P}$
is open (resp.~G$_\delta$, dense, residual).

The next two propositions are crucial in the proof of Theorem \ref{Thmint:main}.
The first proposition has been established in \cite{Suz}.
We refer the reader to \cite{Suz} for the proof.
\begin{Prop}[\cite{Suz}, Corollary 2.4]\label{Prop:amefre}
Freeness and amenability are G$_\delta$-dense in $\mathcal{S}$.
\end{Prop}
The second proposition is a variant of Proposition 3.3 in \cite{Suz}.
\begin{Prop}\label{Prop:min}
The following property is G$_\delta$-dense in $\mathcal{S}$.
The restriction to $\Gamma_S$ is minimal.

\end{Prop}
\begin{proof}
Fix an enumeration $S=(s(n))_{n=1}^\infty$ of $S$.
For any proper clopen subset $V$ of $X$,
we show that the following property is open dense in $\mathcal{S}$.
\begin{quote}$\mathcal{M}(V)$: the $\Gamma_S$-saturation $\Gamma_S\cdot V$ of $V$ fulfills $X$.
\end{quote}
It is not hard to check that the property $\mathcal{M}(V)$ is open in $\mathcal{S}$.
(Cf.~the proof of Proposition 3.3 in \cite{Suz}.)
To show the density of $\mathcal{M}(V)$ in $\mathcal{S}$, take a $\varphi \in \Homeo(X)$
satisfying $\varphi(V)=X \setminus V$.
Let $\alpha \in \mathcal{S}$ be given.
For each $n \in \mathbb{N}$, we define $\alpha^{(n)} \in \mathcal{S}$ to be
\[ \alpha^{(n)}_{u} := \left\{ \begin{array}{ll}
 \varphi & {\rm for~} u=s(n),\\
 \alpha_u & {\rm for~} u \in U \setminus \{s(n)\}.\\
 \end{array} \right.\]

The sequence $(\alpha^{(n)})_{n=1}^\infty$ then converges to $\alpha$ in $\mathcal{S}$.
Obviously each $\alpha^{(n)}$ satisfies $\mathcal{M}(V)$.
Therefore the property $\mathcal{M}(V)$ is open dense in $\mathcal{S}$.

Now, thanks to the Baire category theorem, the conjunction $\bigwedge_{V} \mathcal{M}(V)$
is G$_\delta$-dense in $\mathcal{S}$. Obviously this property is equivalent to the stated property.
\end{proof}
We remark that the proofs of Propositions \ref{Prop:amefre} and \ref{Prop:min} themselves
can avoid the Baire category theorem.
However, to combine these two propositions, we need to apply the Baire category theorem.
Here we record the required result as a corollary.
\begin{Cor}\label{Cor:generic}
There is an amenable free action of $\Gamma$ on $X$
whose restriction to $\Gamma_S$ is minimal.
\end{Cor}
The author does not know concrete examples of $\Gamma$-actions
with these properties.
\section{Proofs of Main results}
In this section, we complete the proofs.
We first recall the following basic observation.
\begin{Lem}\label{Lem:norm}
Let $G$ be a locally compact non-amenable group.
Let $A$ be a $G$-\Cs-algebra.
Assume that $A$ has a nonzero $G$-ideal $I$
whose $G$-action is of the form
$\ad\circ\varphi$ for some strictly continuous homomorphism
$\varphi \colon G \rightarrow \mathcal{U}(M(I))$.
Then $A \rtimes G \neq A \rc G$.
\end{Lem}
\begin{proof}
Since $C_{\rm c}(G, I)$ is an ideal of $C_{\rm c}(G, A)$,
every $\ast$-representation of $C_{\rm c}(G, I)$ on a Hilbert space
extends to a $\ast$-representation of $C_{\rm c}(G, A)$.
Consequently the canonical map $I \rtimes G \rightarrow A \rtimes G$ is injective.
It thus suffices to show the claim in the case $I=A$.
Let $\tau$ and $\alpha$ denote the trivial $G$-action and the given $G$-action on $A$ respectively.
Define a map $\Phi \colon C_{\mathrm c}(G, A) \rightarrow C_{\mathrm c}(G, A)$
to be
$\Phi(f)(s):= f(s)\varphi(s)$ for $f\in C_{\mathrm c}(G, A)$, $s\in G$.
Then $\Phi$ defines a $\ast$-isomorphism
between the twisted convolution algebras of the $G$-\Cs-algebras $(A, \tau)$ and $(A, \alpha)$.
The $\Phi$ extends to $\ast$-isomorphisms
of their reduced and full \Cs-completions.
Obviously the induced $\ast$-isomorphisms are compatible with the canonical quotient maps.
The claim now follows from Hulanicki's Theorem \cite{Hul}.
\end{proof}

\begin{proof}[Proof of Theorem \ref{Thmint:main}]
Let $G$ be a second countable locally compact non-amenable group.
Choose $\alpha \in \mathcal{S}$ as in Corollary \ref{Cor:generic}.
Take a homomorphism
$\varphi\colon \Gamma \rightarrow G$ onto a dense subgroup of $G$
which annihilates $\Gamma_S$.
Denote $\lambda$ and $\rho$ the left and right translation action of $G$ on itself respectively.
Define $\beta := \alpha \otimes (\lambda \circ \varphi)$
and $A:=C_0(X\times G) \fc{\beta} \Gamma.$
Observe that $\beta$ is minimal
because of the choices of $\alpha$ and $\varphi$.
Also, being an extension of $\alpha$, the $\beta$ is free and amenable.
It therefore follows from Proposition 4.8 of \cite{Ana87} and the Corollary in page 124 of \cite{AS} that $A$ is simple.
Note that the action $\beta$ commutes with the action
$\tilde{\rho}:=\id_{C(X)} \otimes \rho$ of $G$.
Hence $\tilde{\rho}$ induces a $G$-action $\gamma$ on $A$.

Proposition 3.3 of \cite{Tak} yields the relations
\[C_0(G) \fc{\rho} G= C_0(G) \rca{\rho} G \cong \mathbb{K}(L^2(G)).\]
We therefore obtain the following canonical identifications.
\begin{eqnarray*} A\fc{\gamma} G &=& (C_0(X\times G) \rtimes \Gamma)\fc{\gamma }G\\
&=& (C_0(X\times G) \fc{\tilde{\rho}} G)\rtimes \Gamma\\
&=& (C_0(X\times G) \rca{\tilde{\rho}}G)\rtimes\Gamma.
\end{eqnarray*}
By Proposition 4.8 of \cite{Ana87}, we obtain
\[(C_0(X\times G) \rca{\tilde{\rho}} G)\rtimes\Gamma=(C_0(X\times G) \rca{\tilde{\rho}} G)\rc \Gamma.\]
This yields $A \fc{\gamma} G = A \rca{\gamma} G$ as desired.

We now construct an ambient $G$-\Cs-algebra $B$ of $A$ as in the statement.
Observe that the left and right regular representations of $G$
induce the $G$-actions on $\mathbb{K}(L^2(G))$ via the adjoint.
We denote these actions by $\lambda$ and $\rho$ respectively.
Note that these two actions commute.
Let \[\iota \colon C_0(G) \rightarrow M(\mathbb{K}(L^2(G)))=\mathbb{B}(L^2(G))\]
denote the multiplication action of $C_0(G)$ on $L^2(G)$.
Then $\iota$ is $G$-equivariant with respect to both the left and right actions.
Let $\eta := \lambda \circ \varphi$.
Then we equip $C(X)\otimes \mathbb{K}(L^2(G))$ with the $\Gamma$-action $\alpha \otimes \eta$.
Put $B:=(C(X)\otimes \mathbb{K}(L^2(G))) \rtimes \Gamma$.
Since $\mathbb{K}(L^2(G))$ is simple and $\alpha$ is minimal, free, and amenable,
it follows from Proposition 4.8 of \cite{Ana87} and the Corollary in page 122 of \cite{AS} that $B$ is simple.
As $\rho$ commutes with $\eta$,
the $G$-action $\id_{C(X)}\otimes \rho$ induces
a $G$-action $\zeta$ on $B$.
Then notice that $\iota$ induces a non-degenerate $G$-embedding
$\Phi \colon A \rightarrow M(B)$.
By Lemma \ref{Lem:norm},
we obtain
\begin{eqnarray*} 
B \rtimes G &=& ((C(X)\otimes \mathbb{K}(L^2(G)))\rtimes G) \rtimes \Gamma\\
&\neq&((C(X)\otimes \mathbb{K}(L^2(G)))\rc G)\rc \Gamma =B \rc G.
\end{eqnarray*} 
\end{proof}
For a locally compact group $G$ and a $G$-\Cs-algebra $A$,
let $u \colon G \rightarrow \mathcal{U}({A \rc G})$ denote the canonical embedding.
We refer to the $G$-action $\ad \circ u$ as the conjugate action of $G$ on $A \rc G$.
\begin{proof}[Proof of Proposition \ref{Propint:unitalsimple}]
Let $G$ be a discrete exact group.
Then one can find an amenable minimal free action $\alpha$ of $G$ on a compact Hausdorff space $Y$ (cf.~\cite{BO}, \cite{Oza}).
Put $C:=C(Y)\rc G$ and $A:= C^{\otimes \mathbb{N}}$.
Let $\beta$ denote the conjugate action of $G$ on $C$.
Define $\gamma := \beta^{\otimes \mathbb{N}}$.
We set $B:= A \rc G$. Let $\eta$ denote the conjugate action of $G$ on $B$.
Note that $\eta|_A = \gamma$. We show that the inclusion $A \subset B$ satisfies the desired conditions.

Note first that the canonical conditional expectation $E\colon B=A \rc G \rightarrow A$ is $G$-equivariant (see Remark 4.1.10 of \cite{BO} for the detail).
The Corollary in page 124 of \cite{AS} shows the simplicity of $C$.
Hence $A$ is simple.
We next observe that $A$ can be realized as
the $G$-inductive limit of $G$-\Cs-algebras
$A_n:=C^{\otimes n} \otimes C(Y), n\in \mathbb{N}.$
By Proposition 4.8 of \cite{Ana87},
each $A_n$ satisfies
$A_n \rtimes G = A_n \rc G$.
Hence $A\rtimes G = A \rc G$.
The Corollary in page 122 of \cite{AS} shows the simplicity of
$A_n\rc G$ for all $n\in \mathbb{N}$.
This yields the simplicity of $B= A \rc G$.
By Lemma \ref{Lem:norm}, we obtain $B \rtimes G \neq B \rc G$.
\end{proof}
\begin{Rem}\label{Rem:tower}
By modifying the above construction,
one can form an increasing sequence
\[A_1 \subset B_1 \subset A_2 \subset B_2 \subset \cdots \subset A_n \subset B_n \cdots\] of unital simple $G$-\Cs-algebras satisfying
$A_n \rtimes G =A_n \rc G$ and $B_n \rtimes G \neq B_n \rc G$ for all $n \in \mathbb{N}$.
Here we only describe how to construct $A_2$ and $B_2$;
the rest of the construction is the same.
Put $A_1:= A$ and $B_1:=B$, the $G$-\Cs-algebras in the proof of Proposition \ref{Propint:unitalsimple}.
Define $A_2 := {B_1^{\otimes \mathbb{N}}}$, equipped with the diagonal $G$-action.
We identify $B_1$ with the first tensor component of $A_2$.
Put $B_2:= A_2 \rc G$, equipped with the conjugate $G$-action.
Then one can confirm the required relations by the similar argument
to the proof of Proposition \ref{Propint:unitalsimple}.
\end{Rem}
\begin{Rem}\label{Rem:Cuntz}
Thanks to Kirchberg's $\mathcal{O}_2$-absorption theorem (Theorem 3.8 of \cite{KP}),
one can construct $G$-\Cs-algebras $A, B$ in Theorem \ref{Thmint:main}
to further satisfy the relations $A= \mathcal{O}_2 \otimes \mathbb{K} \cong A \rtimes G \cong B$.
Similarly, in Proposition \ref{Propint:unitalsimple}, under the additional assumption that $G$ is countable, one can choose $G$-\Cs-algebras in the statement
to further satisfy $A=\mathcal{O}_2  \cong A \rtimes G \cong B$,
and the analogous statement holds true for Remark \ref{Rem:tower}.
(To make $G$-\Cs-algebras separable, one has to take the amenable $G$-space
$Y$ in the proof to be second countable.)
\end{Rem}
\begin{proof}[Proof of Proposition \ref{Propint:functorialu}]
By Theorem A of \cite{BCL}, one can choose an amenable action $\alpha$ of $G$ on a compact Hausdorff space $Y$.
We set $A:= C(Y^{\mathbb{Z}}) \fc{\sigma} \mathbb{Z}$,
where $\sigma$ denotes the shift automorphism.
As the diagonal $G$-action $\beta:=\alpha^{\otimes \mathbb{Z}}$ commutes with $\sigma$,
the $\beta$ extends to a $G$-action $\gamma$ on $A$.
We then obtain
\begin{eqnarray*} 
A\fc{\gamma} G &=& (C(Y^{\mathbb{Z}}) \fc{\beta} G) \fc{\tilde{\sigma}} \mathbb{Z}\\
&=& (C(Y^{\mathbb{Z}}) \rca{\beta} G) \fc{\tilde{\sigma}} \mathbb{Z}\\
&=& A \rca{\gamma} G,
\end{eqnarray*}
where $\tilde{\sigma}$ denotes the automorphism induced from $\sigma$.
It is not hard to check that $Z(A)=\mathbb{C}$.

Now take a non-degenerate faithful $\ast$-representation $\varphi \colon A \rightarrow \mathbb{B}(\mathfrak{H})$
of $A$ on a Hilbert space $\mathfrak{H}$.
Set $\mathfrak{K}:= \mathfrak{H} \otimes L^2(G)$.
Define $\tilde{\varphi}\colon A \rightarrow \mathbb{B}(\mathfrak{K})$
to be
\[(\tilde{\varphi}(a)( \xi \otimes g))(s):= g(s){\alpha_{s}^{-1}}(a)\xi \]
for $a\in A$, $\xi \in \mathfrak{H}$, $g\in C_{\rm c}(G)$, $s\in G$.
(More precisely, the formula defines an element in $C_{\mathrm c}(G, \mathcal{H})$
which we identify with a dense subspace of $\mathfrak{K}$ in the obvious way.)
Define $B$ to be the \Cs-algebra on $\mathfrak{K}$
generated by $\tilde{\varphi}(A)$ and $\mathbb{K}(\mathfrak{K})$.
Let $\tilde{\lambda}:=\id_{\mathfrak{H}} \otimes \lambda$ be the left translation action
of $G$ on $\mathfrak{K}$.
The composite $\ad\circ \tilde{\lambda}$ then defines a (continuous) $G$-action on $B$ which makes the map $\tilde{\varphi}$ $G$-equivariant.
From now on we identify $A$ with the $G$-\Cs-subalgebra $\tilde{\varphi}(A)$ of $B$.
Observe that $A\cap \mathbb{K}(\mathfrak{K})=0$.
We therefore obtain a $G$-short exact sequence
\[ 0 \longrightarrow \mathbb{K}(\mathfrak{K}) \longrightarrow B {\xrightarrow{~ E ~}} A \longrightarrow 0.\]
Since $Z(A)=\mathbb{C}$ and $Z(\mathbb{K}(\mathfrak{K}))=0$,
a diagram chasing shows that $Z(B)=\mathbb{C}$. The quotient map $E$
provides a $G$-conditional expectation of
the inclusion $A \subset B$.
Lemma \ref{Lem:norm} yields $B\rtimes G \neq B \rc G$.
Therefore the inclusion $A\subset B$ satisfies the required conditions.
\end{proof}
\begin{Rem}
By modifying the above construction,
for any locally compact non-amenable exact group $G$,
one can show that the condition $-\rtimes G \neq -\rc G$
is not inherited to $G$-inductive limits.
Here we sketch the construction.
Let $B$ be the $G$-\Cs-algebra in the proof of Proposition \ref{Propint:functorialu}.
Put $D:=B^{\otimes \mathbb{N}}$ (equipped with the diagonal $G$-action).
Then, on the one hand, by a similar argument to the proof of Proposition \ref{Propint:unitalsimple},
one can show that $D \rtimes G = D \rc G$.
On the other hand, the $D$ can be realized as the $G$-inductive limit of $G$-\Cs-algebras
$D_n:= B^{\otimes n}, n \in \mathbb{N}$.
Lemma \ref{Lem:norm} yields  $D_n \rtimes G \neq D_n \rc G$ for all $n \in \mathbb{N}$.
\end{Rem}
\subsection*{Acknowledgements}
The author is grateful to Jianchao Wu for a stimulating question, which motivates
an original idea of Proposition \ref{Propint:functorialu}, during
the conference at the University of Glasgow (C Star), 2014.
This work was supported by JSPS KAKENHI Grant-in-Aid for Young Scientists
(Start-up) (No.~17H06737) and tenure track funds of Nagoya University.

\end{document}